\documentclass[11pt]{amsart}
\usepackage{amsmath, amssymb}
\usepackage{amsfonts}
\usepackage{graphicx}
\usepackage{mathrsfs}
\usepackage[arrow,matrix,curve,cmtip,ps]{xy}
\usepackage{amsthm}
\usepackage{enumerate}
\usepackage{color}
\usepackage[colorlinks]{hyperref}
\usepackage{tikz}

\allowdisplaybreaks

\newtheorem{thm}{Theorem}[section]
\newtheorem{prop}[thm]{Proposition}
\newtheorem{cor}[thm]{Corollary}
\newtheorem{lem}[thm]{Lemma}

\newtheorem*{theorem*}{Theorem}
\theoremstyle{remark}
\newtheorem{rem}[thm]{Remark}
\newtheorem{defn}[thm]{Definition}
\newtheorem{example}[thm]{Example}

\numberwithin{equation}{section}

\newcommand{\N}{\mathbb{N}}

\newcommand{\C}{\mathbb{C}}
\newcommand{\La}{\Lambda}
\newcommand{\la}{\lambda}
\newcommand{\g}{\mathcal{G}}
\newcommand{\ag}{A_K(\mathcal{G})}
\newcommand{\go}{\mathcal{G}^{(0)}}
\newcommand{\kp}{\mathrm{KP}_K(\Lambda)}
\newcommand{\ac}{A_{\mathbb{C}}(\mathcal{G})}

\begin{document}
\title[Purely infinite Steinberg Algebras]{Non-simple purely infinite Steinberg Algebras with applications to Kumjian-Pask algebras}

\author{Hossein Larki}

\address{Department of Mathematics\\ Faculty of Mathematical Sciences and Computer\\ Shahid Chamran University of Ahvaz, Iran}
\email{h.larki@scu.ac.ir}


\date{\today}



\begin{abstract}
We characterize properly purely infinite Steinberg algebras $\ag$ for strongly effective, ample Hausdorff groupoids $\g$. As an application, when $\La$ is a strongly aperiodic $k$-graph, we show that the notions of pure infiniteness and proper pure infiniteness are equivalent for the Kumjian-Pask algebra $\kp$, which may be determined by the proper infiniteness of vertex idempotents. In particular, for unital cases, we give simple graph-theoretic criteria for the (proper) pure infiniteness of $\kp$.

Furthermore, since the complex Steinberg algebra $A_\C(\g)$ is a dense subalgebra of the reduced groupoid $C^*$-algebra $C^*_r(\g)$, we focus on the problem that ``when does the proper pure infiniteness of $A_\C(\g)$ imply that of $C^*_r(\g)$ in the $C^*$-sense?". In particular, we show that if the Kumjian-Pask algebra $\mathrm{KP}_\C(\La)$ is purely infinite, then so is $C^*(\La)$ in the sense of Kirchberg-R{\o}rdam.
\end{abstract}

\maketitle


\section{Introduction}

Inspired by the Kirchberg-R${\o}$rdam's definition of non-simple purely infinite $C^*$-algebras \cite{kirch00}, Aranda Pino, Goodearl, Perera, and Siles Molina introduced the notion of purely infinite and properly purely infinite properties for (not necessarily simple) rings \cite{pin10} as the nonsimple generalization of \cite{ara02} (see Section 3 for definitions). In particular, it is shown in \cite[Proposition 3.17]{pin10} that if a $C^*$-algebra is purely infinite in the sense of \cite{pin10}, then it is purely infinite in the Kirchberg-R${\o}$rdam's sense. These concepts are closely related, such that for s-unital rings, the properly purely infinite property implies the purely infinite one \cite[Lemma 3.4(i)]{pin10}, and we know that the converse is also true if the ring is either exchange \cite[Corollary 2.9]{pin10} or simple \cite[Corollary 5.8]{pin10} for example.

Let $K$ be a field. Following \cite{cla16}, in this paper we work with strongly effective ample groupoids $\g$ and associated Steinberg algebras $A_K(\g)$. Our aim here is to investigate (non-simple) purely infinite Steinberg algebras and Kumjian-Pask algebras. Recall that the theory of Steinberg algebras were independently introduced in \cite{ste10} and \cite{cla14} as a pure algebraic analogue to Renault's groupoid $C^*$-algebras \cite{ren80}. There has been a lot of interest in the study of Steinberg algebras, partly because they include, and give a groupoid approach for, many interesting classes of algebras such as Leavitt path \cite{abr05}, Kumjian-Pask \cite{pin13}, and inverse semigroup algebras \cite{ste10} among others. Moreover, many structural properties of these algebras can be described by that of the underlying groupoids; such as, simplicity \cite{bro14}, primeness \cite{ste18-2}, primitivity \cite{ste16}, simple pure infiniteness \cite{bro19}, and ideal structure \cite{cla16}, among others.

The contents of this article are as follows. We review in Section 2 some preliminaries about groupoids, Steinberg algebras and Kumjian-Pask algebras which will be used throughout. In Section 3, we prove that a Steinberg algebra $\ag$ is properly purely infinite if and only if its characteristic functions on compact open bisections are properly infinite. Then we give a sufficient condition for this property. In Section 4, we apply this result to the Kumjian-Pask algebras of strongly aperiodic $k$-graphs, and derive interesting consequences. In particular, we show that purely infinite and properly purely infinite Kumjian-Pask algebras coincide, which may be characterized by verifying the proper infiniteness of their vertex idempotents.  This is the generalization of \cite[Theorem 7.4]{pin10} which is for the class of Leavitt path algebras, but our proof, even in the Leavitt path algebras setting, is different and simpler. Moreover, we see that a unital Kumjian-Pask algebras is (properly) purely infinite if and only if every vertex in the underlying $k$-graph receives at least one nontrivial path.

For any locally compact Hausdorff groupoid $\g$, we know that the complex Steinberg algebra $A_\C(\g)$ is a dense subalgebra of the reduced groupoid $C^*$-algebra $C^*_r(\g)$ (see \cite[Proposition 4.2]{cla14}). In \cite{bro19}, the authors considered \cite[Problem 8.4]{pin10} for simple groupoid algebras, and proved that for any second-countable ample groupoid $\g$, if the complex Steinberg algebra $A_\C(\g)$ is purely infinite simple, then so is the $C^*$-algebra $C^*_r(\g)$ in the sense of \cite{cun81}. In Section 5, we consider second-countable and amenable groupoids $\g$ and  generalize this result to not necessarily simple groupoid algebras with a shorter proof. To do this, we first show that a groupoid $C^*$-algebra $C^*_r(\g)$ is $C^*$-purely infinite (in the sense of \cite{kirch00}) if and only if all its characteristic functions are $C^*$-properly infinite, which is interesting in its own right. Since the goupoid $\g_\La$ associated to each $k$-graph $\La$ is always amenable and second-countable, we conclude that the pure infiniteness of complex Kumjian-Pask algebra $\mathrm{KP}_\C(\La)$ implies that of the $C^*$-algebra $C^*(\La)$ for strongly aperiodic $\La$'s.


\section{Preliminaries}

Here, we review some definitions and basic properties of groupoids, Steinberg algebras and Kumjian-Pask algebras.

\subsection{Groupoids}

By {\it groupoid}, we mean a small category $\g$ endowed with a composition $\g^{(2)}\subseteq \g\times \g\rightarrow \g$, by $(\alpha,\beta)\mapsto \alpha\beta$, and an inversion $\alpha\mapsto \alpha^{-1}$ on $\g$ satisfying the following conditions:
\begin{enumerate}
  \item if $(\alpha,\beta),(\beta,\gamma)\in \g^{(2)}$, then both $(\alpha\beta,\gamma)$ and $(\alpha,\beta\gamma)$ are composable and we have $(\alpha\beta)\gamma=\alpha(\beta\gamma)$;
  \item $(\alpha^{-1},\alpha)\in \g^{(2)}$ for all $\alpha\in \g$; and
  \item if $(\alpha,\beta)\in \g^{(2)}$, then $\alpha^{-1}(\alpha\beta)=\beta$ and $(\alpha\beta)\beta^{-1}=\alpha$.
\end{enumerate}
Then the source and range maps on $\g$ are $s(\alpha)=\alpha^{-1}\alpha$ and $r(\alpha)=\alpha\alpha^{-1}$ for all $\alpha\in \g$. In particular, we have $(\alpha,\beta)\in \g^{(2)}$ if and only if $s(\alpha)=r(\beta)$. For every two subsets $A,B$ of $\g$, we define the product of $A$ and $B$ as
$$AB:=\{\alpha\beta:\alpha\in A,\beta\in B, s(\alpha)=r(\beta)\}.$$
We will let $\go$ the {\it unit space of $\g$}, that is $\go=\{\alpha^{-1}\alpha:\alpha\in \g\}$. Moreover, the {\it isotropy subgroupoid} of $\g$ is defined by
$$\mathrm{Iso}(\g):=\{\alpha\in \g:s(\alpha)=r(\alpha)\}.$$

We say that a subset $U\subseteq \go$ is {\it invariant} if, for $\alpha\in \g$, $s(\alpha)\in U$ implies $r(\alpha)\in U$, under which we have $r(s^{-1}(U)=U=s(r^{-1}(U))$. Note that in case $U$ is an invariant subset of $\go$, then $\g_U:=s^{-1}(U)$ is a subgroupoid of $\g$ (with the unit space $U$), which coincides with the restriction
$$\g|_U:=\{\alpha\in \g:s(\alpha),r(\alpha)\in U\}.$$

A groupoid $\g$ endowed with a topology is called a {\it topological groupoid} in case the composition and inverse maps are continuous. Let $\g$ be a topological groupoid. A subset $B\subseteq \g$ is a {\it bisection} if the restrictions $r|_B$ and $s|_B$ are homeomorphisms. Then we say that $\g$ is {\it ample} if $\g$ has a basis of compact open bisections. In this case, $\go$ is an open and totally disconnected subset of $\g$.

\begin{defn}
A topological groupoid $\g$ is called {\it effective} if the interior of $\mathrm{Iso}(\g)$ is $\go$. Also, we say that $\g$ is {\it strongly effective} in case for every nonempty closed invariant subset $V\subseteq \go$, the restricted groupoid $\g_V$ is effective.
\end{defn}

Clearly, when $\g$ is strongly effective, then it is effective as well because $\go$ is a closed invariant set. In this paper, we only consider the strongly effective ample Hausdorff groupoids.

\subsection{Steinberg Algebras}

Let $\g$ be an ample groupoid and $K$ a field. The {\it Steinberg algebra} $\ag$ associated to $\g$ consists of all locally constant and compactly supported functions from $\g$ into $K$. Then it has $K$-algebra structure by considering the pointwise addition as usual and the convolution multiplication as
$$(fg)(\alpha)=\sum_{r(\beta)=r(\alpha)}f(\beta)g(\beta^{-1}\alpha) \hspace{5mm} (f,g\in \ag ~~ \mathrm{and} ~~ \alpha\in \g).$$
Since $\g$ is assumed to be ample, we have
$$\ag=\mathrm{span}_K\left\{1_B: B \mathrm{~is~a~compact~open~bisection}\right\},$$
where $1_B$ is the characteristic function on $B$. In particular, $\ag$ has local units, in the sense that for each $f\in \ag$ there is an idempotent $p\in \ag$ such that $fp=pf=f$, and so is s-unital. Examples of Steinberg algebras of ample groupoids include Leavitt path algebras \cite{cla14} as well as Kumjian-Pask algebras \cite{cla14,cla17}(see Section 4 for details).

By {\it ideal}, we always mean a two-sided one. According to \cite[Theorem 3.1]{cla16}, when $\g$ is a strongly effective ample Hausdorff groupoid, then every ideal of $\ag$ is of the form
$$I_U:=\left\{f\in \ag:\mathrm{supp} f\subseteq \g_U\right\},$$
for some open invariant subset $U$ of $\go$. Moreover, if $U$ is an open invariant subset of $\go$ and $D:=\go\setminus U$, then \cite[Lemma 3.6]{cla16} says that the restrictive map $q(f)=f|_{\g_D}$ is an epimorphism from $\ag$ into $A_K(\g_D)$ such that $\ker(q)=I_U$; thus we have $\ag/I_U\cong A_K(\g_D)$. We will use this result in Section 3 because the property of proper infiniteness is closely related to the structure of ideals and quotients.

\subsection{Higher-rank graphs and their Kumjian-Pask algebras}

Here, we recall some basic definitions of graphs of rank $k\geq 1$ and associated Kumjian-Pask algebras from \cite{kum00,pin13,cla14-2}, which will be needed in the results of Section 4. Let $\N=\{0,1,2,\ldots\}$. For fixed integer $k\geq 1$, we consider the additive semigroup $\N^k$ with the identity $0:=(0,\ldots,0)$. We write each element $n\in \N^k$ as $n=(n_1,\ldots,n_k)$, and the generators of $\N^k$ by $e_1,\ldots,e_k$, where $e_{ij}=\delta_{i,j}$ for $1\leq i,j\leq k$. We may put a partial order $\leq$ on $\N^k$ by $m\leq n$ $\Longleftrightarrow$ $m_i\leq n_i$ for all $1\leq i\leq k$, and denote $m\vee n$ and $m\wedge n$ for the coordinate-wise maximum and minimum, respectively.

\begin{defn}[\cite{kum00}]
A  \emph{$k$-graph} (or \emph{higher-rank graph}) is a countable small category $\La=(\La^0,\La,r,s)$, where $\La^0$ is the objects, $\La$ morphisms, $r,s:\La\rightarrow \La^0$ the range and source maps, which is equipped with a \emph{degree functor} $d:\La \rightarrow \N^k$ satisfying the \emph{unique factorisation property}: for each $\la\in\La$ and $0\leq n\leq d(\la)$, there exists a unique factorisation $\la=\la(0,n) \la(n,d(\la))$ for $\la$ such that $d(\la(0,n))=n$, $d(\la(n,d(\la)))=d(\la)-n$.
\end{defn}

Every $k$-graph $\La$ has an edge-colored graph as its 1-skeleton, so we usually like to refer the objects of $\La^0$ as vertices and the morphisms of $\La$ as paths. If $H\subseteq \La^0$ and $E\subseteq \La^1$, we write
$$HE:=\{\la\in E: r(\la)\in H\}  \hspace{5mm} \mathrm{and} \hspace{5mm} E H:=\{\la\in E: s(\la)\in H\},$$
and when $H=\{v\}$ is a singleton, we simply write $vE$ and $E v$, respectively. For any $n\in \N^k$, we define $\La^n:=\{\la\in \La:d(\la)=n\}$ and
$$\Lambda^{\leq n}:=\{\lambda\in \Lambda : d(\lambda)\leq n, ~\mathrm{and}~~  d(\lambda)+e_{i}\leq n \Longrightarrow s(\lambda)\Lambda^{e_{i}}=\emptyset\}.$$
Given $\mu,\nu\in \La$, the set of \emph{minimal common extensions} for $\mu$ and $\nu$ is denoted by $\mathrm{MCE}(\mu,\nu)$, which is
$$\mathrm{MCE}(\mu,\nu):=\left\{\la\in \La^{d(\mu)\vee d(\nu)}: \mu\alpha=\la=\nu\beta ~~\mathrm{for~ some~} \alpha,\beta\in\La\right\}.$$

A $k$-graph $\La$ is called \emph{row-finite} if $v\La^n$ is finite for all $v\in \La^0$ and $n\in \N^k$. Also, we say $\La$ is \emph{locally convex} in case for every $\la\in \La^{e_i}$ and $1\leq j\neq i \leq k$, $r(\la)\La^{e_j}\neq \emptyset$ implies $s(\la)\La^{e_j}\neq \emptyset$ \cite[Definition 3.10]{rae03}. In this article, we consider only locally convex and row-finite $k$-graphs.

\begin{example}
Let $m\in (\N\cup\{\infty\})^k$ for $k\geq 1$. Then the set $\Omega_{k,m}=\{(p,q)\in\Bbb N^{k}\times\Bbb N^{k}: p\leq q\leq m\}$ equipped with the degree map as $d(p,q)=q-p$, and the range and source maps as $r(p,q)=(p,p)$ and $s(p,q)=(q,q)$ is a row-finite $k$-graph.
\end{example}

A {\it boundary path} of $\La$ is a degree preserving functor $x: \Omega_{k,m}\rightarrow\Lambda$ such that for every $p\in \N^{k}$ and $1\leq i\leq k$, $p\leq m$ and $p_{i}=m_{i}$ imply that $x(p,p)\Lambda^{e_{i}}=\emptyset$. Then we define $d(x)=m$ and usually call $x(0,0)$ {\it range of $x$}, written by $r(x)$. The set of all boundary paths of $\La$ is denoted by $\Lambda^{\leq \infty}$.

There are some definitions for aperiodicity in the literature, which are equivalent for locally convex row-finite $k$-graphs. Following \cite{rae03}, we say that $\La$ is {\it aperiodic} if for each $v\in \La^0$, there exists $x\in v\La^{\leq \infty}$ such that $\alpha\neq \beta\in \La v$ implies $\alpha x\neq \beta x$.

\begin{defn}[{\cite{pin13}}]
Let $\Lambda$ be a locally convex, row-finite $k$-graph and $K$  a field. The {\it Kumjian-Pask algebra} associated to $\La$ is the universal $K$-algebra $\kp$ generated by a family $\{s_\la,s_{\la^*}:\la\in \La\}$ (which is called a {\it Kumjian-Pask $\La$-family}) satisfying the following conditions:
\begin{enumerate}[(KP1)]
  \item $s_v s_w=\delta_{v,w} s_v$ for all $v,w\in\La^0$.
  \item $s_\la s_\mu=s_{\la \mu}$ and $s_{\mu^*}s_{\la^*}=s_{(\la \mu)^*}$ for all $\la,\mu\in\La$ with $s(\la)=r(\mu)$.
  \item $s_{\la^*} s_\mu=\delta_{\la,\mu}s_\la$ for all $\la,\mu\in \La$.
  \item $s_{v}=\sum_{\lambda\in v\Lambda^{\leq n}}s_{\lambda}s_{\lambda^{\ast}}$ for all $v\in \Lambda^0$ and $n\in \N^k$.
\end{enumerate}
\end{defn}
It is show in \cite[Theorem 3.4]{pin13} and \cite[Theorem 3.7(a)]{cla14-2} that the $K$-algebra $\kp$ exists and is unique up to isomorphism.

A subset $H\subseteq \La^0$ is called {\it hereditary} if for every $\la\in \La$, $r(\la)\in H$ implies $s(\la)\in H$, and is called {\it saturated} if for every $v\in \La^0$ and $n\in \N^k$, $\{s(\la): \la\in v\La^{\leq n}\}\subseteq H$ implies $v\in H$. Note that if $H$ is a saturated hereditary subset of $\La^0$, then $\La\setminus \La H$ is a locally convex row-finite $k$-graph \cite[Theorem 5.2 (b)]{rae03}. We say that $\La$ is {\it strongly aperiodic} in case all such quotient $k$-graphs are aperiodic. Recall from \cite[Corollary 5.7]{pin13} that, when $\La$ is strongly aperiodic, every (two-sided) ideal of $\kp$ is of the form
$$I_H:=\mathrm{span}_K\left\{s_\la s_{\mu^*}: \la,\mu\in \La, s(\la)=s(\mu)\in H\right\}$$
for some saturated hereditary subset $H$ of $\La^0$ (see also \cite[Theorem 9.4]{cla14-2}), and we have $\kp/I_H\cong \mathrm{KP}_K(\La\setminus \La H)$ \cite[Proposition 5.5]{pin13}.


\section{Properly purely infinite Steinberg algebras}

In this section, we characterize properly purely infinite Steinberg algebras $\ag$ by giving necessary and sufficient conditions for it. Our result is the algebraic analogue of \cite[Theorem 4.1]{bon17} and the non-simple version of \cite[Theorem 3.2]{bro19}. Since the concept of proper purely infiniteness is closely related to the ideal and quotient structure \cite[Corollary 2.9]{pin10}, we consider here strongly effective groupoids to apply the results of \cite[Section 3]{cla16}. However, verifying this concept for general groupoids seems to be more complicated, even for the $k$-graph's ones.

Let us first recall some definitions and terminology from \cite{pin10}. Suppose that $R$ is a ring. If $a\in M_m(R)$ and $b\in M_n(R)$ are two squire matrices over $R$, we denote
$$a\oplus b:=\left(
                                                                                                               \begin{array}{cc}
                                                                                                                 a & 0 \\
                                                                                                                 0 & b \\
                                                                                                               \end{array}
                                                                                                             \right) \in M_{m+n}(R).$$
We define $a\precsim b$ if there exist $x\in M_{m,n}(R)$ and $y\in M_{n,m}(R)$ such that $a=xby$. Note that we may consider $a$, $b$, $x$ and $y$ as elements in $M_{m+n}(R)$ by enlarging them by zeros. For $a,b\in R$, we write $a\sim b$ if there exist $r,s\in R$ such that $rs=a$ and $sr=b$, $a\leq b$ if $ab=ba=a$, and $a<b$ if $a\leq b$ and $a\neq b$. A nonzero element $a\in R$ is called {\it infinite} in case $a\oplus x\precsim a$ for some $x\in R$, and {\it properly infinite} in case $a\oplus a\precsim a$ (or equivalently, $a\oplus a\precsim a\oplus 0$). As explained after \cite[Remark 2.7]{pin10}, if $p\in R$ is an idempotent, $p$ is infinite if and only if there is a subidempotent $p< q$ such that $p\sim q$.

\begin{lem}\label{lem3.1}
Let $R$ be a ring and let $p$ and $q$ be two idempotents in $R$. If $p\sim q$ and $p$ is properly infinite, then so is $q$.
\end{lem}

\begin{proof}
Suppose that $p=xy$ and $q=yx$ for some $x,y\in R$. Let $p\oplus p=ApB$ where $A\in M_{2,1}(R)$ and $B\in M_{1,2}(R)$. Then we have
\begin{align*}
q\oplus q&=q^2\oplus q^2=(y\oplus y)(xy\oplus xy)(x\oplus x)\\
&=(y\oplus y)(p\oplus p)(x\oplus x)\\
&=(y\oplus y)(ApB) (x\oplus x)\\
&=\big((y\oplus y)Ax\big)q\big(yB (x\oplus x)\big),
\end{align*}
and therefore $q\oplus q\precsim q$ as desired.
\end{proof}

\begin{defn}[{\cite[Definition 3.1]{pin10}}]\label{defn3.2}
Let $R$ be a ring. We say that $R$ is purely infinite if
\begin{enumerate}
  \item each quotient of $R$ is not a division ring, and
  \item for every $a\in R$, $b\in RaR$ implies $b\precsim a$.
\end{enumerate}
Also, we say that $R$ is {\it properly purely infinite} in case all nonzero elements of $R$ are properly infinite.
\end{defn}

According to \cite[Lemma 3.4]{pin10}, every properly purely infinite s-unital ring is purely infinite as well, and we know at least that the converse is also true for exchange cases \cite[Corollary 5.8]{pin10}. Moreover, it follows from \cite[Theorem 7.4]{pin10} that these notions are equivalent for the class of Leavitt path algebras, and in the next section, we will prove a similar result for the class of Kumjian-Pask algebras.

We use the following lemma to prove Theorem \ref{thm3.4} below.

\begin{lem}\label{lem3.3}
Let $\mathcal{G}$ be an ample Hausdorff groupoid and $K$ a field. Then the following are equivalent.
\begin{enumerate}
  \item $\g$ is effective.
  \item Every nonzero ideal of $\ag$ contains a nonzero idempotent $1_V$ for some compact open subset $V\subseteq \go$.
  \item Every nonzero one-sided ideal of $\ag$ contains a nonzero idempotent $p$ such that $p\sim 1_V$ for some compact open subset $V\subseteq \go$.
\end{enumerate}
\end{lem}

\begin{proof}
(1) $\Longrightarrow$ (3). Let $I$ be a nonzero right ideal of $\ag$. Fix a nonzero element $a\in I$. Then, by \cite[Lemma 2.2]{cla15}, we may write
$$a=\sum_{B\in \mathcal{F}}r_B 1_B$$
where $\mathcal{F}$ is a finite collection of mutually disjoint compact open bisections. We may follow the proof of \cite[Theorem 3.2]{bro19} to find an idempotent $p\in I$ and a compact open $V\subseteq \go$ such that $p\sim 1_V$, as desired. For left ideals, we can  argue  similarly.

(3) $\Longrightarrow$ (2). Let $I$ be a nonzero ideal of $\ag$. According to (3), $I$ contains a nonzero element $p$ such that $p\sim 1_V$ for some compact open $V\subseteq \go$. If $p=xy$ and $1_V=yx$ for some $x,y\in \ag$, then we have
$$1_V=1_V^2=y(xy)x=ypx\in I,$$
giving (2).

(2) $\Longrightarrow$ (1). This follows from \cite[Lemma 3.4]{cla16}.
\end{proof}

\begin{thm}\label{thm3.4}
Let $\g$ be a strongly effective, ample and Hausdorff groupoid and let $K$ be a field. Suppose that $\mathcal{B}$ is a basis of compact open sets for $\go$. Then the following are equivalent:
\begin{enumerate}
  \item $\ag$ is properly purely infinite.
  \item For every $V\in \mathcal{B}$, $1_V$ is properly infinite.
  \item For every invariant open subset $U\subseteq \go$ and any nonempty relatively compact open subset $V\subseteq D:=\go\setminus U$, $1_V$ is an infinite idempotent in $A_K(\g_D)$.
  \item Every nonzero one-sided ideal in any quotient of $\ag$ contains an infinite idempotent.
\end{enumerate}
In particular, if one of the above statements holds (and so all), then $\ag$ is purely infinite as well.
\end{thm}

\begin{proof}
(1) $\Longrightarrow$ (2) is immediate.

(2) $\Longrightarrow$ (3). Let $U$ be an invariant open subset of $\go$ and write $D:=\go\setminus U$. Given any relatively compact open $\widetilde{V}\subseteq D$, there is an open subset $V\subseteq \go$ such that $\widetilde{V}=V\cap D$. We may also suppose that $V$ is compact. Indeed,  since $\mathcal{B}$ is a basis for $\go$, we can write $V=\bigcup_{C\in\mathcal{C}}C$ for some subcollection $\mathcal{C}\subseteq \mathcal{B}$. Then $\widetilde{V}=\bigcup_{C\in\mathcal{C}}C\cap D$. As $\widetilde{V}$ is compact, there are finitely many $C_1,\ldots,C_n$ in $\mathcal{C}$ such that $\widetilde{V}\subseteq \bigcup_{i=1}^n C_n\cap D$. Thus, we can replace $V$ with $\bigcup_{i=1}^n C_n$ and assume that $V$ is a compact and open subset of $\go$.

Now part (2) says that $1_V$ is a properly infinite idempotent in $\ag$. Thus, since $1_{\widetilde{V}}$ is the image of $1_V$ in the quotient $A_K(\g_D)\cong \ag/I_U$, \cite[Corollary 2.9]{pin10} implies that $1_{\widetilde{V}}$ is an infinite idempotent in $A_K(\g_D)$.

(3) $\Longrightarrow$ (4). Let $I$ be an ideal of $\ag$. By \cite[Theorem 3.1]{cla16}, there is an open invariant $U\subseteq \go$ such that $I=I_U$, and we have $\ag/I\cong A_K(\g_D)$ for $D:=\go\setminus U$ \cite[Lemma 3.6]{cla16}.

Now fix an arbitrary nonzero one-sided ideal $J$ of $A_K(\g_D)$. Since $\g_D$ is effective, Lemma \ref{lem3.3} implies that $J$ contains a nonzero idempotent $p$ such that $p\sim 1_V$ for some relatively compact open subset $V$ of $D$. Since $1_V$ is infinite in $A_K(\g_D)$ \cite[Corollary 2.9]{pin10}, then so is $p$ as desired.

(4) $\Longrightarrow$ (1) follows from \cite[Proposition 3.13]{pin10}, completing the proof.
\end{proof}

\begin{defn}[{\cite[Definition 2.1]{ana97}}]
Let $\g$ be an ample Hausdorff groupoid. We say that $\g$ is \emph{locally contracting} if for every compact open $V\subseteq \go$, there exists a compact open bisection $B\subseteq \g$ such that
$$s(B)\subsetneq r(B)\subseteq V.$$
\end{defn}

Note that if $s(B)\subsetneq r(B)$, then we have
$$1_{s(B)}=1_{B^{-1}B}=1_{B^{-1}}1_B  \hspace{5mm} \mathrm{and} \hspace{5mm} 1_{r(B)}=1_{B B^{-1}}=1_B 1_{B^{-1}}.$$
Using $1_{s(B)} 1_{r(B)}=1_{s(B)}=1_{r(B)} 1_{s(B)}$, this follows
$1_{s(B)}\lneq 1_{r(B)}\sim 1_{s(B)}$, giving that $1_{r(B)}$ is an infinite idempotent in $\ag$.

\begin{cor}\label{cor3.6}
Let $\g$ be a strongly effective, ample Hausdorff groupoid and let $K$ be a field. If for every open invariant $U\subseteq \go$, the groupoid $\g_{\go\setminus U}$ is locally contractive, then $\ag$ is properly purely infinite.
\end{cor}

\begin{proof}

Let $U$ be an invariant open subset of $\go$, and let $V$ be a relatively compact open subset of $D:=\go\setminus U$. Since $\g_D$ is locally contractive, there is some compact open bisection $B\subseteq \go$ such that $s(B)\subsetneq r(B)\subseteq V$. Hence $1_{r(B)}$ is infinite in $A_K(\g_D)$ as explained above, and by $1_{r(B)}\leq 1_V$, then so is $1_V$ as well. Now (3) $\Rightarrow$ (1) in Theorem \ref{thm3.4} concludes the result.
\end{proof}


\section{Applications to Kumjian-Pask algebras}

According to \cite[Proposition 4.3]{cla14} and \cite[Proposition 5.4]{cla17}, every Kumjian-Pask algebra $\kp$ may be considered as the Steinberg algebra $A_K(\g_\La)$ where $\g_\La$ is the boundary path groupoid of $\La$. So, in view of Theorem \ref{thm3.4}, we may investigate the proper pure infiniteness of Kumjian-Pask algebras by groupoid approach. Moreover, \cite[Theorem 7.4]{pin10} says that the notions of purely infinite and properly purely infinite are equivalent for Leavitt path algebras. Here, we will also generalize this result to the class of Kumjian-Pask algebras by a quite different proof.

Let us first follow \cite[Section 2]{kum00} and briefly review the construction of the boundary path groupoid $\g_\La$ associated to $\La$. (Although the groupoid $\g_\La$ of \cite[Section 2]{kum00} is constructed for a row-finite $k$-graph with no sources, however we may easily generalize it for any row-finite $k$-graph.)  So, suppose that $\La$ is a fixed locally convex row-finite $k$-graph. We set
$$\g_\La:=\left\{(\la x,d(\la)-d(\mu), \mu x): \la,\mu\in\La, s(\la)=s(\mu), x\in s(\la)\La^{\leq \infty} \right\}$$
and define the range and source maps by $r(x,m,y):=x$ and $s(x,m,y):=y$. If we consider the composition as $(x,m,y)(y,n,z):=(x,m+n,z)$ and the inversion as $(x,m,y)^{-1}:=(y,-m,x)$, then $\g_\La$ is a groupoid. We usually identify the unit space $\g_\La^{(0)}$ with the boundary path space $\La^{\leq\infty}$ by $(x,0,x)\leftrightarrow x$. Moreover, we may equip $\g_\La$ with a locally compact Hausdorff topology induced from that of $\La^{\leq \infty}$. Indeed, for each $\la,\mu\in\La$ with $s(\la)=s(\mu)$ we set
$$Z(\la*_s \mu):=\left\{(\la x,d(\la)-d(\mu),\mu x):x\in s(\la)\La^{\leq \infty}\right\}.$$
When $\la=\mu$, we write simply $Z(\la)$ for $Z(\la*_s\la)$. Then \cite[Proposition 2.8]{kum00} shows that the collection $\{Z(\la*_s \mu):\la,\mu\in \La, s(\la)=s(\mu)\}$ forms a basis of compact open bisections for a Hausdorff topology on $\g_\La$, making it an ample groupoid. Moreover, an application of the graded uniqueness theorem implies that $\kp\cong A_K(\g_\La)$, which maps $s_\la$ to $1_{Z(\la*_s s(\la))}$ and $s_{\la^*}$ to $1_{Z(s(\la)*_s \la)}$ (cf. \cite [Proposition 4.3]{cla14} and \cite[Proposition 5.4]{cla17}).

Before Theorem \ref{thm4.2}, we state a lemma.

\begin{lem}\label{lem4.1}
Let $R$ be a purely infinite ring and let $p$ be an idempotent in $R$. If there are mutually orthogonal idempotents $q_1,q_2\in R$ such that $p\in Rq_i R$, then $p$ is properly infinite.
\end{lem}

\begin{proof}
Since $R$ is purely infinite, $p\in Rq_i R$ implies $p\precsim q_i$ for $i\in \{1,2\}$. Suppose $p=a_i q_i b_i$. Putting $x_i:=q_ib_i a_i q_i$, then $x_1,x_2$ are orthogonal idempotents in $R$ such that $x_1\sim p\sim x_2$. Note that if $x_1=rs$ and $x_2=sr$ for $r,s\in R$, then $x_2=x_2^2=s(rs)r=rx_1 s$. Hence $x_1+x_2$ lies in $Rx_1 R$, and the pure infiniteness forces $x_1+x_2\precsim x_1$. Moreover, for
$A=\left(
     \begin{array}{c}
       x_1 \\
       x_2 \\
     \end{array}
   \right)$
and
$B=\left(
     \begin{array}{cc}
       x_1 & x_2 \\
     \end{array}
   \right)
$,
we have $AB=x_1\oplus x_2$ and $BA=x_1+ x_2$. Therefore, we get
$$p\oplus p\sim x_1\oplus x_2\sim x_1 +x_2\precsim x_1 \sim p,$$
establishing $p\oplus p\precsim p$.
\end{proof}

\begin{thm}\label{thm4.2}
Let $\La$ be a locally convex row-finite $k$-graph and $K$ a field. Suppose that $\La$ is strongly aperiodic. Then the following are equivalent.
\begin{enumerate}
  \item $\kp$ is purely infinite.
  \item $\kp$ is properly purely infinite.
  \item For each vertex $v\in \La^0$, $s_v$ is a properly infinite idempotent.
\end{enumerate}
\end{thm}

\begin{proof}
(1) $\Longrightarrow$ (3). Fix $v\in \La^0$. In view of \cite[Corollary 2.9]{pin10}, it suffices to show that the image of $s_v$ in any quotient of $\kp$ is either zero or infinite. Moreover, since $\La$ is strongly aperiodic, every ideal of $\kp$ is of the form of $I_H$ for some saturated hereditary subset $H$ of $\La^0$ \cite[Theorem 9.4]{cla14-2}. So, let $I_H$ be an ideal of $\kp$ such that $v\notin H$. Writing $\Gamma:=\La\setminus H\La$, then $\Gamma$ is an aperiodic $k$-graph and we have $\kp/I_H\cong \mathrm{KP}_K(\Gamma)$, which is purely infinite as well. In the following, for each $\gamma\in \Gamma$, we write $t_\gamma:=s_\gamma+I_H$ in $\mathrm{KP}_K(\Gamma)$. We divide our discussion into two cases:

Case I: There are two distinct vertices $v\neq u,w\in \Gamma^0$ such that $v\Gamma w, u\Gamma w\neq \emptyset$. Fix some $\mu_1\in v\Gamma w$ and $\mu_2\in u\Gamma w$. Then $q_1:=t_{\mu_1}t_{\mu_1^*}$ and $q_2:=t_{\mu_2}t_{\mu_2^*}$ are two orthogonal idempotents in $\mathrm{KP}_K(\Gamma)$ (because $s_v$ and $s_u$ are) such that
$$q_1\sim t_{\mu_1^*}t_{\mu_1}=t_w=t_{\mu_2^*}t_{\mu_2}\sim q_2.$$
Now, Lemma \ref{lem4.1} implies that $t_w$ is properly infinite. Since $t_v\geq q_1\sim t_{\mu^*_1} t_{\mu_1}=t_w$, apply \cite[Corollary 3.9(iii)]{pin10} to conclude that $t_v$ is infinite in $\mathrm{KP}_K(\Gamma)$.

Case II: In the other case, there exists $w\in \Gamma^0$ with $v\Gamma w\neq \emptyset$ such that $w\Gamma u=\emptyset$ for all $u\in \Gamma^0\setminus\{w\}$. If $w\Gamma w=\{w\}$, i.e. $w$ receives no nontrivial paths, then the ideal $I_w=\langle t_w\rangle$ in $\mathrm{KP}_K(\Gamma)$ is a matricial algebra (cf. \cite[Lemma 6.2]{lar18}), which is impossible. Otherwise, since $\Gamma$ is aperiodic, \cite[Lemma 6.1]{lar18} implies that $w$ is the base of an (initial) cycle with an entrance. Thus $t_w$ is infinite (see Lemma \ref{lem4.4}(1) below). Now if $\mu\in v\Gamma w$, then $t_v\geq t_\mu t_{\mu^*}\sim t_{\mu^*}t_\mu=t_w$, and hence $t_v$ would be infinite in each case. Since $I_H$ was arbitrary, this follows that $s_v$ is a properly infinite idempotent in $\kp$ as desired.

(3) $\Longrightarrow$ (2). To prove this part, we consider the groupoid $\g_\La$ associated to $\La$, and then apply Theorem \ref{thm3.4}. According to \cite[Proposition 2.8]{kum00}, the collection $\mathcal{B}=\{Z(\la): \la\in\La\}$ is a basis of compact and open sets for the topology on $\go$. Also, for each $\la\in\La$, we have
$$1_{Z(\la)}=s_\la s_{\la^*}\sim s_{\la^*} s_\la=s_{s(\la)},$$
and hence $1_{Z(\la)}$ is properly infinite by Lemma \ref{lem3.1}. Since $\g_\La$ is strongly effective \cite[Proposition 6.3]{cla17}, Theorem \ref{thm3.4} follows that $\kp=A_K(\g_\La)$ is properly purely infinite.

(2) $\Longrightarrow$ (1) is immediate by \cite[Lemma 3.4(i)]{pin10}.
\end{proof}

Now, analogous to \cite[Theorem 5.4]{lar18}, we may use the notion of generalized cycle to give a criterion under which the vertex idempotents $s_v$ are properly infinite. It is the $k$-graphic version of Corollary \ref{cor3.6}.

\begin{defn}[\cite{eva12}]
A \emph{generalized cycle in $\La$} is a pair $(\mu,\nu)$ of distinct paths $\mu,\nu\in \La$ with $s(\mu)=s(\nu)$ and $r(\mu)=r(\nu)$ such that $\mathrm{MCE}(\mu\tau,\nu)\neq \emptyset$ for all $\tau \in s(\mu)\La$; or equivalently $Z(\mu)\subseteq Z(\nu)$. Note that if $\mu$ is a cycle in $\La\setminus\La^0$ (in the sense $s(\mu)=r(\mu)$), then $(\mu^i,\mu^j)$ is a generalized cycle for every $j\leq i\in \N$. We say that $\tau\in s(\nu)\La$ is an \emph{entrance} for $(\mu,\nu)$ if $\mathrm{MCE}(\mu,\nu\tau)=\emptyset$ (i.e.,  $\nu\tau\La^{\leq \infty}\subseteq Z(\nu)\setminus Z(\mu)$).
\end{defn}

Following \cite{lar18}, if there is a path from $s(\mu)=s(\nu)$ to $v\in \La^0$ we say that \emph{$(\mu,\nu)$ connects to $v$} or \emph{$v$ is reached from $(\mu,\nu)$}.

\begin{lem}\label{lem4.4}
Let $(\mu,\nu)$ be a generalized cycle in $\La$ with an entrance. Then
\begin{enumerate}
  \item $s_{s(\mu)}=s_{s(\nu)}$ is an infinite idempotent in $\kp$.
  \item If $v\in \La^0$ is reached from $(\mu,\nu)$, then $s_v$ is infinite in $\kp$.
\end{enumerate}
\end{lem}

\begin{proof}
\cite[Lemma 4.3]{lar18} says that $s_\mu s_{\mu^*}<s_\nu s_{\nu^*}$. Thus by
$$s_\nu s_{\nu^*}\sim s_{\nu^*}s_\nu =s_{s(\mu)}=s_{\mu^*}s_\mu \sim s_\mu s_{\mu^*},$$
$s_\nu s_{\nu^*}$ is an infinite idempotent, and so is $s_{\nu^*}s_\nu=s_{s(\nu)}=s_{s(\mu)}$ as well.

Statement (2) follows from (1). Indeed, if $\gamma\in v\La s(\nu)$ then $s_v\geq s_\gamma s_{\gamma^*}\sim s_{\gamma^*} s_{\gamma}=s_{s(\mu)}$, and hence $s_\gamma s_{\gamma^*}$ and $s_v$ are infinite by part (1).
\end{proof}

\begin{cor}
Let $\La$ be a strongly aperiodic, locally convex and row-finite $k$-graph and let $K$ be a field. Suppose that for each saturated hereditary $H\subseteq \La^0$, every vertex $v\in \La^0\setminus H=(\La\setminus H\La)^0$ is reached from a generalized cycle with an entrance in the $k$-graph $\La\setminus H\La$. Then $\kp$ is properly purely infinite.
\end{cor}

\begin{proof}
Fix $v\in \La^0$. For any ideal $I_H\trianglelefteq\kp$ with $v\notin I_H$, $v$ is reached from a generalized cycle with an entrance in $\La\setminus H\La$, so $s_v+I_H$ is an infinite idempotent in $\kp/I_H\cong \mathrm{KP}(\La\setminus H\La)$ by Lemma \ref{lem4.4}(2). Therefore, $s_v$ is properly infinite in $\kp$, and Theorem \ref{thm4.2} concludes the result.
\end{proof}

In the case that there are only finitely many vertices connecting to each vertex $v\in \La^0$, we can establish simple graph-theoretic criteria for the pure infiniteness of Kumjian-Pask algebras $\kp$. Note that this case covers all unital Kumjian-Pask algebras.

\begin{prop}\label{prop4.6}
Let $\La$ be a locally convex row-finite $k$-graph and $K$ a field. Suppose that $\La$ is strongly aperiodic, and that $\La^0_{\geq v}:=\{w:v\La w\neq \emptyset\}$ is finite for every $v\in \La^0$. Then the following are equivalent. \begin{enumerate}
  \item $v\La\neq \{v\}$ for every vertex $v\in \La^0$.
  \item Every vertex of $\La$ is reached from a cycle.
  \item $\kp$ is properly purely infinite.
  \item $\kp$ is purely infinite.
\end{enumerate}
\end{prop}

\begin{proof}
(1) $\Longrightarrow$ (2). Fix $v\in \La^0$ and assume $|\La^0_{\geq v}|=t_0$. Using statement (1), for each $t\geq t_0^k$ in $\mathbb{N}$, there exists $\la\in v\La$ such that $|\la|:=d(\la)_1+\cdots+ d(\la)_k=t$, where $d(\la)=(d(\la)_1,\ldots,d(\la)_k)$. This forces $\la(m)=\la(n)$ for some $m<n\in \mathbb{N}^k$, and thus the cycle $\la(m,n)$ connects to $v$.

(2) $\Longrightarrow$ (3). Fix some $v\in \La^0$ and an ideal $I_H$ of $\kp$ such that $s_v\notin I_H$. Then $v\notin H$. For convenience, we write $\Gamma:=\La\setminus \La H$, which is a strongly aperiodic $k$-graph too. We first claim $v\Gamma\neq \{v\}$. Indeed, if $v\Gamma=\{v\}$ then for each $n\in \N^k\setminus\{0\}$ we have $\{r(\la):\la\in v\La^{\leq n}\}\subseteq H$. Thus $v\in H$ by the saturated property, a contradiction. Now implication (1) $\Longrightarrow$ (2) above yields that $v$ is reached from a cycle in $\Gamma$, and using \cite[Lemma 6.1]{lar18}, is also from an (initial) cycle with an entrance. Lemma \ref{lem4.4} says that $s_v+I_H$ is an infinite idempotent in $\mathrm{KP}_K(\Gamma)=\kp/I_H$. Therefore, since $I_H$ was arbitrary, $s_v$ is properly infinite in $\kp$, and consequently $\kp$ is properly purely infinite by Theorem \ref{thm4.2}.

(3) $\Longleftrightarrow$ (4) are obtained from Theorem \ref{thm4.2}.

(3) and (4) $\Longrightarrow$ (1). If $v\La=\{v\}$ for some $v\in\La^0$, then the ideal $I_v=\langle s_v \rangle$ in $\kp$ is isomorphic to the matrix algebra $M_{|v\La|}(K)$ (cf. \cite[Lemma 6.2]{lar18}). But this is impossible because the (proper) pure infiniteness passes to ideals.

\end{proof}

\section{Complex Steinberg aplgebras and groupoid $C^\ast$-algebras}

In the last section, we investigate the relationship between the purely infinite property of Steinberg algebras and groupoid $C^*$-algebras. More precisely, since the complex Steinberg algebra $A_\C(\g)$ may be regarded as a dense $*$-subalgebra of $C^*_r(\g)$ (see \cite[Proposition 4.2]{cla14}), we focus on the conjecture \cite[Problem 8.4]{pin10} for groupoid algebras: If $A_\C(\g)$ is properly purely infinite, is this property inherited by $C^*_r(\g)$? It is shown in \cite[Section 4]{bro19} that this conjecture is true for simple groupoid algebras (when $\g$ is minimal and effective). However, in the nonsimple cases, the properly purely infinite property (in both algebraic and $C^*$-algebraic senses) is closely related to ideal structure and quotients. Here, we work with amenable groupoids in the sense of \cite{ana00} to apply \cite[Proposition 6.1.8]{ana00} and the ideal description of \cite[Corollary 5.9]{bro14}. Moreover, since the groupoid $\g_\La$ associated to each $k$-graph $\La$ is always amenable, we deduce also an interesting result for $k$-graph algebras.

We recall the definition of reduced $C^*$-algebra $C^*_r(\g)$ from \cite{ren80}. Let $\g$ be an ample groupoid. Write $C_c(\g)$ for the complex vector space consisting of compactly supported continuous functions on $\g$, which is an $*$-algebra with the convolution multiplication and the involution $f^*(\alpha):=\overline{f(\alpha^{-1})}$. Then $C^*_r(\g)$ is the completion of $C_c(\g)$ for the left regular representations on $\ell^2(\g)$. Indeed, for each $u\in \go$ and $\g_u:=s^{-1}(\{u\})$, if $\pi_u:C_c(\g)\rightarrow B(\ell^2(\g_u))$ is the left regular $*$-representation, defined by
$$\pi_u(f)\delta_\alpha:=\sum_{s(\beta)=r(\alpha)}f(\beta) \delta_{\beta \alpha} \hspace{5mm} (f\in C_c(\g),~\alpha\in \g_u),$$
then the {\it reduced $C^*$-algebra $C^*_r(\g)$} is the completion of $C_c(\g)$ under the reduced $C^*$-norm
$$\|f\|_r:=\sup_{u\in \go}\|\pi_u(f)\|.$$
Moreover, there is a full $C^*$-algebra $C^*(\g)$ associated to $\g$, which is the completion of $C_c(\g)$ taken over all $\|.\|_{C_c(\g)}$-decreasing representations of $\g$. So, $C^*_r(\g)$ is a quotient of $C^*(\g)$, and \cite[Proposition 6.1.8]{ana00} follows that they are equal if the underlying groupoid $\g$ is amenable.

Let $A$ be a $C^*$-algebra. For positive matrices $a\in M_m(A)^+$ and $b\in M_n(A)^+$ over $A$, we write $a\precsim_{C^*} b$ if there exists a sequence $\{x_n\}$ in $M_{n,m}(A)$ such that $x_n^* b x_n\rightarrow a$. As before, for $a,b\in A$ we denote $a\oplus b:=diag(a,b)$ in $M_{2}(A)$. A nonzero positive element $a\in A$ is called {\it $C^*$-infinite} if $a\oplus b\precsim_{C^*} a$ for some positive $0\neq b\in A$, and is called {\it $C^*$-properly infinite} if $a\oplus a\precsim_{C^*} a$. Following \cite{kirch00}, we say that $A$ is {\it $C^*$-purely infinite} in case every nonzero positive element of $A$ is properly infinite.

\begin{defn}[{See \cite[Chaper 2]{ren80}}]
A groupoid $\g$ is called \emph{topologically principal} in case the set $\{u\in \go:\{\alpha\in \g:r(\alpha)=s(\alpha)\}=\{u\}\}$ is dense in $\go$. We say that $\g$ is \emph{essentially principal} if for every nonempty closed invariant subset $D$ of $\go$, the groupoid $\g_D$ is topologically principal.
\end{defn}

Note that, when $\g$ is second countable, \cite[Proposition 3.3]{ren08} follows that $\g$ is strongly effective if and only if it is essentially principal.

The following is the $C^*$-analogue of Theorem \ref{thm3.4} (see also \cite[Theorem 4.1]{bon17}).

\begin{prop}\label{prop5.2}
Let $\g$ be a second countable Hausdorff ample groupoid and let $\mathcal{B}$ be a basis for $\go$ containing compact open sets. Suppose that $\g$ is amenable (in the sense of \cite{ana00}) and essentially principal. Then $C^*(\g)=C^\ast_r(\g)$ is $C^\ast$-purely infinite if and only if $1_V$ is $C^\ast$-properly infinite for every $V\in \mathcal{B}$.
\end{prop}

\begin{rem}
If $\g$ is second countable, amenable and essentially principal (strongly effective), then \cite[Corollary 5.9]{bro14} implies that $C_0(\go)$ separates closed ideals of $C^*(\g)$ in the sense that for every closed ideals $I\nsubseteq J$ of $C^*(\g)$, there exists $f\in C_0(\go)$ such that $f\in I\setminus J$. Thus, in this case, \cite[Proposition 2.14]{pas07} implies that $C^*$-purely infinite, $C^*$-weakly purely infinite, and $C^*$-strongly purely infinite $C^*$-algebras coincide.
\end{rem}

\begin{proof}[Proof of Proposition \ref{prop5.2}.]
Since the ``only if" implication is immediate, we prove the ``if" part only. Suppose that all $1_V$, for $V\in \mathcal{B}$, are $C^\ast$-properly infinite in $C^*(\g)$. Using \cite[Proposition4.7]{kirch00}, it suffices to to show that every nonzero hereditary $C^*$-subalgebra in any quotient
of $C^*(\g)$ contains an infinite projection. For this, let $I$ be a closed ideal of $C^*(\g)$ and let $A$ be a nonzero hereditary $C^*$-subalgebra of $C^*(\g)/I$. By \cite[Corollary 5.9]{bro14}, there is an open invariant $U\subseteq \go$ such that $I=I_U$, and whence $C^*(\g)/I\cong C^*(\g_{D})$ where $D:=\go\setminus U$. So we may assume $A\subseteq C^*(\g_D)$.

Let $0\neq a\in A$ be a positive element. Then \cite[Lemmma 3.2]{bro15} says that there exists a nonzero positive element $h\in C_0(D)$ such that $h\precsim_{C^*} a$ in $C^*(\g_D)$. Now since $\{V\cap D:V\in \mathcal{B}\}$ is a basis for the induced topology on $\go_D$, there is $V\in \mathcal{B}$ such that $V\cap D\neq \emptyset$ and $h(x)>0$ for all $x\in \widetilde{V}:=V\cap D$. Consider $g\in C_0(D)$ defined by
$$g(x)=\left\{
         \begin{array}{ll}
           \frac{1}{\sqrt{h(x)}} & x\in \widetilde{V} \\
           0 & x\in D\setminus \widetilde{V}.
         \end{array}
       \right.
$$
Then we have $ghg=1_{\widetilde{V}}$, and thus $1_{\widetilde{V}}\precsim h\precsim a$ in $C^*(\g_D)$. By the argument after \cite[Proposition 2.6]{kirch00}, we may find $y\in C^*(\g_D)$ such that $1_{\widetilde{V}}=y^* a y^*$. If we set $z:=a^{1/2} y$, then $z^* z=y^* a y=1_{\widetilde{V}}$ and $zz^*=a^{1/2}yy^* a^{1/2}\in A$, giving $zz^*\sim_{C^*}1_{\widetilde{V}}$. Note  that $1_{\widetilde{V}}$ is (properly) infinite because it is the image of $1_V$ in the quotient $C^*(\g_D)\cong C^*(\g)/I$. Therefore, $zz^*$ is an infinite projection in $A$, as desired.
\end{proof}

We now prove the main result of this section, which is the nonsimple generalization of \cite[Theorem 4.7 and Corollary 4.8]{bro19} with an easier way.

\begin{prop}\label{prop5.4}
Let $\g$ be a second countable Hausdorff ample groupoid. Suppose that $\g$ is amenable and strongly effective. If $A_{\mathbb{C}}(\g)$ is properly purely infinite (in the sense of Definition \ref{defn3.2}), then $C^*(\g)=C^*_r(\g)$ is $C^*$-purely infinite.
\end{prop}

\begin{proof}
Suppose that $\ac$ is properly purely infinite. According to Proposition \ref{prop5.2}, it suffices to show that $1_V$ is $C^*$-properly infinite for every nonempty compact open $V\subseteq \go$. So fix some such $\emptyset\neq V\subseteq \go$. Theorem \ref{thm3.4} says that $1_V$ is a properly infinite idempotent in $\ac$. Thus there exists $a,b\in M_2(\ac)$ such that $a(1_V\oplus 0)b=1_V\oplus 1_V$. We may consider $M_2(\ac)$ as a subalgebra of $M_2(C^*(\g))$ in the natural way, and then apply \cite[Proposition 2.4 (iii) $\Rightarrow$ (ii)]{ror92} to find a sequence $\{r_i\}\subseteq M_2(C^*(\g))$ such that $r_i(1_V\oplus 0)r_i^*\rightarrow 1_V\oplus 1_V$. Therefore, we have $1_V\oplus 1_V\precsim_{C^*} 1_V\oplus 0$, meaning that $1_V$ is $C^*$-properly infinite.
\end{proof}

For any locally convex row-finite $k$-graph $\La$, we may associate a universal $C^*$-algebra $C^*(\La)$ \cite{kum00,rae03}, which coincides with the groupoid $C^*$-algebra $C^*_r(\g_\La)$ (see \cite[Corollary 3.5(i)]{kum00} for example). Hence, Proposition \ref{prop5.4} deduces the following result for $k$-graph algebras.

\begin{cor}
Let $\La$ be a locally convex row-finite $k$-graph, which is strongly aperiodic. If $\mathrm{KP}_\C(\La)$ is purely infinite, then $C^*(\La)$ is $C^*$-purely infinite.
\end{cor}

\begin{proof}
By Theorem \ref{thm4.2}, $A_\C(\g_\La)=\mathrm{KP}_\C(\La)$ is properly purely infinite. Since the groupoid $\g_\La$ is always second-countable and amenable, Proposition \ref{prop5.4} follows the result.
\end{proof}


\end{document}